\documentclass[a4paper,oneside,11pt]{article}

\usepackage{amsmath,amsfonts,amscd,amssymb}
\usepackage{longtable,geometry}
\usepackage[english]{babel}
\usepackage[utf8]{inputenc}
\usepackage[active]{srcltx}
\usepackage[T1]{fontenc}
\usepackage{graphicx}
\usepackage{pstricks}
\usepackage{bbm}
\usepackage{mathtools}

\usepackage{MnSymbol}
\usepackage{stmaryrd}
\usepackage{nicefrac}
\usepackage{calrsfs}

 \usepackage{rotating}

\usepackage{xcolor}
\usepackage{framed}

\colorlet{shadecolor}{blue!15}

\newtheorem{theorem}{Theorem}
\newtheorem{conj}{Conjecture}
\newtheorem{question}{Question}

\newtheorem{lemma}[theorem]{Lemma}
\newtheorem{proposition}[theorem]{Proposition}

\newtheorem{conjecture}[conj]{Conjecture}

\newenvironment{proof}[1][\relax]
  {\paragraph{Proof\ifx#1\relax\else~of #1\fi}}%
  {~\hfill$\square$\par\bigskip}

\newcommand{\bbE}{\mathbb{E}}

\newcommand{\bbG}{\mathbb{G}}

\newcommand{\bbR}{\mathbb{R}}

\newcommand{\bbT}{\mathbb{T}}

\newcommand{\bbV}{\mathbb{V}}

\newcommand{\bbZ}{\mathbb{Z}}

\newcommand{\rk}[1]{\bgroup\color{red}%
  \par\medskip\hrule\smallskip%
  \noindent\textbf{#1}%
  \par\smallskip\hrule\medskip\egroup}

\title{A note on Schramm's locality conjecture for random-cluster models}
\author{Hugo Duminil-Copin\thanks{Institut des Hautes \'Etudes Scientifiques, Universit\'e de Gen\`eve, \texttt{duminil@ihes.fr}}\ \ and
	Vincent Tassion\thanks{ETH Zurich, \texttt{vincent.tassion@math.ethz.ch}}}
\date{\today}

\begin{document}
\maketitle
 
 \begin{abstract}In this note, we discuss a generalization of Schramm's locality conjecture to the case of random-cluster models. We give some partial (modest) answers, and present several related open questions. Our main result is to show that 
the critical inverse temperature of the Potts model on $\bbZ^r\times(\bbZ/2n\bbZ)^{d-r}$ (with $r\ge3$) converges to the critical inverse temperature of the model on $\bbZ^d$ as $n$ tends to infinity. Our proof relies on the infrared bound and, contrary to the corresponding statement for Bernoulli percolation, does not involve renormalization arguments.
\end{abstract}
 \section{Motivation}

 In \cite{bs1}, Benjamini and Schramm initiated the theory of Bernoulli percolation on general transitive graph, thus motivating many new questions and problematics in the field. Among them, one conjecture, now known under the name of  Schramm's locality conjecture, was asked in \cite{bnp}. Roughly speaking, it can be stated as follows: the critical parameter of Bernoulli percolation is continuous in the local topology on transitive locally finite graphs with $p_c<1$ (see below for a formal definition in the context of random-cluster models). Since the formulation of the conjecture, a number of results have been proved \cite{bnp,MarTas,SXZ14}, yet a final answer is still lacking.

We would like to advertise that Schramm's conjecture can be extended to random-cluster model and the associated Potts models. This paper contains the proofs of two results in that direction. The results are modest but we believe that they open some interesting questions. Even though we will not discuss it here, let us mention that in recent years,  Schramm's conjecture was also stated for self-avoiding walks in \cite{ib}, and that some partial results were obtained in \cite{gl1,gl2}. 

\section{About Schramm's conjecture for random-cluster models}

\paragraph{Random-cluster model on transitive graphs}
Through this note, all the graphs are assumed to be connected and locally-finite.  Consider an infinite transitive graph $\bbG=(\bbV,\bbE)$ (recall that a graph is said to be transitive if its group of automorphisms acts transitively on $\bbV$). We  call one of the vertex the {\em origin} and denote it by 0. A percolation configuration $\omega$ on a finite subgraph $G=(V,E)$ of $\bbG$ is a subset of $E$, which will be seen as a subgraph of $G=(V,E)$ with vertex-set $V$ and edge-set $\omega$. Let $|\omega|$ be the number of edges in $\omega$ and $k_1(\omega)$ the number of connected components in $\omega$, when all connected components intersecting the vertex boundary $\partial G:=\{x\in V:\exists y\in\bbV\setminus V\text{ such that }\{x,y\}\in\bbE\}$ are counted as one.

For $p\in[0,1]$ and $q>0$, the random-cluster measure on $G$ with edge-weight $p$ and
  cluster-weight $q$ is defined
by the formula
$$\phi_{G,p,q}^1[\omega]=\frac{1}{Z^1(G,p,q)}\big(\frac p{1-p}\big)^{|\omega|}q^{k_1(\omega)},$$
where $Z^1(G,p,q)$ is  such that the measure has mass one. 

The model may be extended to $\bbG$ by taking weak limits of measures on finite subgraphs $G$ tending to $\bbG$; see \cite[Thm 4.19]{Gri06}. From now on, set $\phi_{\bbG,p,q}^1$ for the measure on $\bbG$. 

Below, we write $0\longleftrightarrow \infty$ to denote the event that the connected component of $0$ in $\omega$ is infinite. Define the {\em critical point} of the model on $\bbG$ by the formula
$$p_c(\bbG,q):=\inf\{p\in[0,1]: \phi_{\bbG,p,q}^1[0\longleftrightarrow\infty]>0\}.$$

\paragraph{Local convergence and Schramm's conjecture}
Given an infinite transitive graph $\bbG$, we consider the ball of radius $R$ (for the graph distance) around the origin 0. Up to isomorphism of rooted graphs, it does not depend on the choice of the origin, and we simply refer to it as the ball of radius $R$ in $\bbG$. We say that a sequence of infinite transitive graphs $(\bbG_n)$ converges (locally) to an infinite transitive graph $\bbG$ if for any $R>0$, there exists $N=N(R)>0$ such that for all $n\ge N$, the balls of radius $R$ in $\bbG_n$ and $\bbG$ are isomorphic (as rooted graphs). Schramm's conjecture for random-cluster models can be stated as follows.
\begin{conjecture}[Schramm's conjecture for random-cluster model]
Fix $q>0$ and consider a sequence of infinite transitive $\bbG_n$ converging to $\bbG$. If $p_c(\bbG_n,q)<1$ for all $n$, then, 
\begin{equation}\lim_{n\rightarrow\infty}p_c(\bbG_n,q)=p_c(\bbG,q).\label{eq:pl}\end{equation}
\end{conjecture} 
For Bernoulli percolation Pete \cite[Section 14]{Pet15} noticed that
\begin{equation}
\liminf_{n\rightarrow\infty}p_c(\bbG_n,q)\ge p_c(\bbG,q)\label{eq:5}
\end{equation}
can be deduced from the mean-field lower bound. This inequality can be obtained in a number of other (elementary) ways for Bernoulli percolation; see e.g.~\cite[Thm 5.3]{Tas14} or \cite[Sec 1.2]{DumTas16} for finite criteria approaches. The same inequality was also known in the $q=2$ case; see \cite{DumTas16}.

For random-cluster models with $q\ge1$, the mean-field lower bound was recently established by the authors and Aran Raoufi \cite{DumRaoTas17}. The first result of this note is to derive \eqref{eq:5} from the mean-field lower bound.
\begin{proposition}\label{prop:1}
Fix $q\ge1$. Consider a sequence of infinite transitive graphs $\bbG_n$ converging to $\bbG$. Then,
$$\liminf_{n\rightarrow\infty}p_c(\bbG_n,q)\ge p_c(\bbG,q).$$
\end{proposition}

\begin{proof}
Fix $q\ge1$ and drop it from the notation. In \cite{DumRaoTas17} (see also \cite{Dum17} for a statement in the nearest neighbor case), one proves that for an infinite transitive graph $\bbG$, any $q\ge1$ and $p\in[0,1]$,
\begin{equation}\label{eq:a}\phi_{\bbG,p}^1[0\longleftrightarrow\infty]\ge c(p-p_c(\bbG)),\end{equation}
where the constant $c>0$ depends a priori on $p$ and $\bbG$, but can be bounded uniformly from below as $p$ remains away from $0$ and $1$, and the degree of $\bbG$ remains bounded ($c$ also depends on $q$ but $q$ is fixed here). 
If the liminf is equal to 1, there is nothing to prove and we now choose $p$ such that 
$$1>p> \liminf_{n\rightarrow\infty}p_c(\bbG_n)=:\tilde p.$$ 
For any fixed $R>r>0$, pick $n$ large enough such that the ball $B_R$ of size $R$ in $\bbG_n$ is the same as in $\bbG$. By comparison between boundary conditions for the random-cluster model \cite[Lemma 4.14]{Gri06}, we deduce that 
\begin{align*}
\phi_{B_R,p}^1[0\longleftrightarrow\partial B_r]&\ge \phi_{\bbG_n,p}^1[0\longleftrightarrow\partial B_r]\ge \phi_{\bbG_n,p}^1[0\longleftrightarrow\infty]\ge c(p-p_c(\bbG_n)).
\end{align*}
(Above, $B_r$ denotes the ball of size $r$ around the origin.) In the last inequality, we used \eqref{eq:a} together with the observation that for every $n$ large enough, the ball of size 1 in $\bbG_n$ is isomorphic to the ball of size 1 in $\bbG$, and therefore the degrees of $\bbG_n$ and $\bbG$ are the same. As a consequence, the constant $c>0$ can be chosen independent of $n$. 

Taking the liminf implies that 
\begin{align*}
\phi_{B_R,p}^1[0\longleftrightarrow\partial B_r]&\ge c(p-\tilde p).
\end{align*}
Letting $R$ tend to infinity, the convergence of $\phi^1_{B_R,p}$ to $\phi^1_{\bbG,p}$  implies that
$$\phi_{\bbG,p}^1[0\longleftrightarrow\partial B_r]\ge c(p-\tilde p).$$
Letting $r$ tend to infinity concludes that 
$$\phi_{\bbG,p}^1[0\longleftrightarrow\infty]\ge c(p-\tilde p)>0,$$
which implies that $p\ge p_c(\bbG)$. The claim follows.
\end{proof}
Exactly as in the case of Bernoulli percolation, the difficult part is to prove that 
$$\limsup_{n\rightarrow\infty}p_c(\bbG_n,q)\le p_c(\bbG,q).$$
In particular, this raises the following natural question.\begin{question}
Extend the locality results known for Bernoulli percolation to the random-cluster model with cluster-weight $q\ge1$.
\end{question}
In particular, the highly non-amenable case would be of great interest.
\paragraph{Quotient graphs}
Quotient graphs appear naturally when studying local limits of graphs. Let $\Gamma$ be a normal subgroup of a group of automorphisms of $\bbG=(\bbV,\bbE)$ acting transitively on $\bbV$. The quotient graph $\bbG/\Gamma$ is the (transitive, locally-finite) graph with vertex-set given by the equivalence classes $\{\Gamma v:v\in \bbV\}$ and edge-set given by the $\{\Gamma u,\Gamma v\}$ such that there exist $u_0\in\Gamma u$ and $v_0\in\Gamma v$ with $\{u_0,v_0\}\in\bbE$.

In \cite{bs1}, Benjamini and Schramm mentioned several questions regarding inequalities between the critical parameters of a graph and its quotient. We would like to highlight the fact that here the inequalities are not obvious at all in our case.
\begin{question}
Is it always the case that $p_c(\bbG,q)\le p_c(\bbG/\Gamma,q)$? If not, find a counter-example. If yes, when is the inequality strict?
\end{question}

In order to go back to our question on locality, notice that we may produce sequences of graphs converging to $\bbG$ by considering quotients by smaller and smaller groups of automorphisms. Actually, when restricted to Cayley graphs of finitely \emph{presented} groups, the local convergence is always by quotient (this is explained in detail in \cite{MarTas} for abelian groups). For this reason, understanding the relation between percolation on a graph and its quotient is important toward a better understanding of Schramm's locality conjecture.

\section{A special case of Schramm's conjecture} 

A simple example of sequence $(\bbG_n)$ converging to a graph $\bbG$ is provided by the graphs $\bbZ^r\times(\bbZ/n\bbZ)^{d-r}$ that converge to $\bbZ^d$ as $n$ tends to infinity. For this graph, we obtain the following result.
\begin{theorem}\label{thm:main}
Fix an integer $q\ge2$ and $d>r\ge3$, then 
$$\lim_{n\rightarrow\infty} p_c(\bbZ^r\times(\bbZ/2n\bbZ)^{d-r},q)=p_c(\bbZ^d,q).$$
\end{theorem}
On the one hand, there are two noteworthy restrictions to our theorem: we do not treat the natural case of $r=2$, and more importantly we are bound to integer values of $q\ge2$. This second restriction is important. Indeed, the random-cluster models with integer values of $q\ge2$ enjoy some additional properties that are wrong for Bernoulli percolation. We believe these additional properties to be very interesting, and we would therefore like to encourage the reader to pursue potential applications (see the discussion below the proof of the theorem).

On the other hand, the proof of the theorem is very short. This is quite surprising since the similar result for Bernoulli percolation (i.e.~$q=1$) is not that simple to obtain: it relies on  Grimmett and Marstrand's celebrated result \cite{GriMar90} (see \cite{lss}  for the explanation of how this result implies the one above for $q=1$).

This also suggests that Schramm's locality conjecture may be simpler to obtain for integer values of $q\ge2$ than for $q=1$. This immediately raises the following question.
\begin{question}
Is Schramm's locality conjecture for a certain value of $q\ge1$ implying the conjecture for other values of $q$?
\end{question}
To motivate this question, let us mention that Benjamini and Schramm \cite{bs1} mentioned several questions regarding the existence of a phase transition for Bernoulli percolation on general graphs, i.e.~whether $p_c(\bbG,1)<1$ or not. It can be easily proved using monotonocity arguments (see \cite[Eq.~(5.5)]{Gri06}) that for all $q>0$, $p_c(\bbG,1)<1$ if and only if $p_c(\bbG,q)<1$. In this case, the original questions of \cite{bs1} can be translated directly into equivalent questions for random-cluster models. 

\section{Proof of Theorem~\ref{thm:main}}
\label{sec:proof-theorem}
We will rely heavily on the connection between the random-cluster model with integer cluster-weight and the Potts model. The Potts model is one of the most fundamental examples of a lattice spin model and studying its properties near its phase transition is an active topic of research; see e.g.~\cite{Dum17} for a recent account. Fix an integer $q\ge2$ and introduce a polyhedron $\Omega\subset\bbR^{q-1}$ with $q$ elements (often interpreted as colors) satisfying that for any ${\rm a},{\rm b}\in \Omega$, ${\rm a}\cdot {\rm b}$ is equal to $1$ if ${\rm a}={\rm b}$ and $-1/(q-1)$ otherwise, where $\cdot$ denotes the scalar product on $\bbR^{q-1}$.

Let $G=(V,E)$ be a finite subgraph of $\mathbb G$ and $\beta>0$. The {\em $q$-state Potts model on
  $G$ at inverse-temperature $\beta>0$ with monochromatic boundary conditions ${\rm b}\in\Omega$} is defined as follows. The {\em energy} of a configuration
$\sigma=(\sigma_x:x\in G)\in \Omega^V$ is given by the Hamiltonian
\begin{equation*}\label{eq:def}H_G^{\rm b}(\sigma)~:=~-\sum_{\substack{x,y\in V\\ \{x,y\}\in E}} \sigma_x\cdot \sigma_y-\sum_{\substack{x\in V,y\notin V\\ \{x,y\}\in \bbE}}\sigma_x\cdot {\rm b}\end{equation*} 
and the probability measure 
$\langle\cdot\rangle_{G,\beta,q}^{\rm b}$ is defined by
$$ \langle X\rangle_{G,\beta,q}^{\rm b}~:=~\frac{\displaystyle \sum_{\sigma\in\Omega^V}X(\sigma)\exp[-\beta H_G^{\rm b}(\sigma)]}{\displaystyle\sum_{\sigma\in\Omega^V}\exp[-\beta H_G^{\rm b}(\sigma)]}$$for every $X:\Omega^V\longrightarrow\bbR$.

The Potts model on $\mathbb G$ can be defined by taking the weak limit of measures on a nested sequence of finite graphs, exactly as for the random-cluster model. The infinite-volume measure is denoted by $\langle\cdot\rangle_{\bbG,\beta,q}^{\rm b}$. The model undergoes a phase transition between absence/existence of long-range order at the so-called critical inverse temperature $\beta_c(\mathbb G,q)$ defined by
$$m(\beta,\bbG,q):=\langle\sigma_0\cdot{\rm b}\rangle_{\bbG,\beta,q}^{\rm b}\begin{cases}=0&\text{ if $\beta<\beta_c(\mathbb  G,q)$,}\\ >0&\text{ if $\beta>\beta_c(\mathbb G,q)$.}\end{cases}$$
The Potts model is coupled to the random-cluster model, see \cite[Sec 1.4]{Gri06} in such a way that 
$$m(\beta,\bbG,q)=\phi^1_{\bbG,p,q}[0\longleftrightarrow\infty]$$
when $\beta=-\tfrac{q-1}q\log(1-p)$.
We therefore deduce that 
\begin{equation}\beta_c(\bbG,q):=-\tfrac{q-1}q\log(1-p_c(\bbG,q))\label{eq:h}\end{equation}
so that Theorem~\ref{thm:main} follows from the following result.
\begin{theorem}\label{thm:main2}
For integers $q\ge2$ and $d>r\ge3$, 
$$\lim_{n\rightarrow \infty} \beta_c(\bbZ^r\times(\bbZ/2n\bbZ)^{d-r},q)~=~\beta_c(\bbZ^d,q).$$
\end{theorem}

The main tool used for the proof of this theorem is the so-called infrared bound. 
Define $\bbT_{N,n}:=(\bbZ/2N\bbZ)^r\times(\bbZ/2n\bbZ)^{d-r}$. Following \cite[Eq.~(3.18)]{Bis09}, introduce  the Green function $\mathsf G_{N,n}$ on the torus $\mathbb T_{N,n}$ by
\begin{equation}\label{eq:2}
\forall x,y\in \mathbb T_{N,n}\qquad  \mathsf G_{N,n}(x,y):=\frac 1{|\mathbb T_{N,n}|}\sum_{\mathbf k\in \mathbb T_{N,n}^*\setminus\{0\}}\frac {e^{\mathbf k\cdot(\mathbf x-\mathbf y)}} {1-\phi(\mathbf k )}, 
\end{equation}
where $\phi(\mathbf k)=\tfrac1d\sum_{j=1}^d \cos(k_j)$ is the characteristic function associated with the simple random walk in dimension $d$, and $$\mathbb T_{N,n}^*:=[\tfrac{2\pi}{N}(\bbZ/2N\bbZ)]^r\times[\tfrac{2\pi}n(\bbZ/2n\bbZ)]^{d-r}.$$
\begin{lemma}\label{cor:estimate}
Consider $\beta>0$ and two integers $n$ and $N$. For any
$v\in\bbR^{\bbT_{N,n}}$ with $\sum_{x}v_{x}=0$,
\begin{equation*}\label{eq:11}\sum_{x,y\in\bbT_{N,n}}v_{x}v_{y}\ \langle \sigma_x\cdot\sigma_y\rangle_{\bbT_{N,n},\beta}~\le ~\tfrac{q-1}{2\beta}\sum_{x,y\in\bbT_{N,n}}v_{x}v_{y}\ \mathsf G_{N,n}(x,y).
\end{equation*}
\end{lemma}
This inequality follows from reflection positivity. It originated in the works \cite{FroSimSpe76,FroSpe81} and has since then been extremely useful in statistical physics. We do not include the proof of this result nor discuss it any further. We refer to \cite{Bis09} for a good review on the subject.

Note that the fact that the widths of the torus in the previous lemma have to be even is the reason why our theorem is restricted to $\bbZ^r\times(\bbZ/2n\bbZ)^{d-r}$ and not simply to $\bbZ^r\times(\bbZ/n\bbZ)^{d-r}$.

Even though we do not use them here, we mention for completeness that \cite{Aba} provides bounds on $\mathsf G_{N,n}(x,y)$ that are valid for finite $N$ and $n$.

\begin{proof}[Theorem~\ref{thm:main2}] We use the notation $\bbT_{N,n}$ from above and set $\bbT_{\infty,n}:=\bbZ^r\times(\bbZ/2n\bbZ)^{d-r}$. 
Below, we fix $\beta$ and $q$ and drop them from the notation. 

Fix a finite subset $E$ of $\bbZ^d$ and see $E$ as a subset of $\bbT_{N,n}$ provided that $N$ and $n$ are sufficiently large to contain it (injectively). Similarly, we see 0 as a vertex of $\bbT_{N,n}$.

Note that $$\sum_{y\in \mathbb T_{N,n}}G_{N,n}(\cdot,y)= \sum_{x\in \mathbb T_{N,n}}G_{N,n}(x,\cdot)=0\quad\text{and}\quad \sum_{y\in \mathbb T_{N,n}}\langle\sigma_0\cdot\sigma_y\rangle_{\bbT_{N,n}}=\sum_{y\in \mathbb T_{N,n}}\langle\sigma_x\cdot\sigma_y\rangle_{\bbT_{N,n}}$$
for every $x,y\in \bbT_{N,n}$. We may therefore apply Lemma~\ref{cor:estimate} in $\bbT_{N,n}$ to $v\in\bbR^{\bbT_{N,n}}$ defined for every $x\in\bbT_{N,n}$ by
 $$v_x=\frac1{|\mathbb T_{N,n}|}-\frac1{|E|}\mathbbm 1[x\in E]$$
to get that
 \begin{align}\label{eq:argh}\underbrace{\frac1{|E|^2}\sum_{\substack{x,y\in E}}\langle \sigma_x\cdot\sigma_y\rangle_{\bbT_{N,n}}}_{\mathsf S_{N,n}}~-~\underbrace{\frac1{|\bbT_{N,n}|}\sum_{y\in \mathbb T_{N,n}}\langle\sigma_0\cdot\sigma_y\rangle_{\bbT_{N,n}}}_{\mathsf T_{N,n}}~\le~ \tfrac{q-1}{2\beta}~ \underbrace{\frac1{|E|^2}\sum_{x,y\in E}\mathsf G_{N,n}(x,y)}_{\mathsf U_{N,n}}.\end{align}

We first wish to take the limit of \eqref{eq:argh} as $N$ and $n$ tend to infinity. We start by the terms on the left. Since $E$ is fixed, the terms $\mathsf S_{N,n}$ and $\mathsf T_{N,n}$ can be treated by taking limits  term by term.
It will be easier to use the random-cluster model to understand the different dominations. Set $p$ such that
$$\beta=-\tfrac{q-1}q\log(1-p).$$
Let $\phi_{\bbT_{N,n}}$ be the random-cluster model on $\bbT_{N,n}$  and $\phi^0_{\bbZ^d}$ be the random-cluster measure on $\bbZ^d$ with free boundary conditions (both with parameters $p$ and $q$) \cite[Thm~4.19]{Gri06}. Note that by coupling (see again \cite[Sec~1.4]{Gri06}) we have
$$\langle\sigma_x\cdot\sigma_y\rangle_{\bbT_{N,n}}=\phi_{\bbT_{N,n}}[x\longleftrightarrow y].$$
Any sub-sequential limit (as $n$ and $N$ tend to infinity) of the family $(\phi_{\bbT_{N,n}})$ is a limit-random-cluster measure on $\bbZ^d$ \cite[Def~4.15]{Gri06} and is therefore stochastically dominating $\phi^0_{N,n}$ \cite[Thm~4.19]{Gri06}. As a consequence, we obtain that
\begin{equation*}
 \liminf_{n\rightarrow\infty} \liminf_{N\to\infty}\ \phi_{\bbT_{N,n}}[x\longleftrightarrow y]\ge \phi^0_{\mathbb Z^d}[x\longleftrightarrow y]\ge \phi^0_{\mathbb Z^d}[x,y\longleftrightarrow \infty]\ge \phi^0_{\mathbb Z^d}[0\longleftrightarrow\infty]^2,
 \end{equation*}
 where in the third inequality we used the uniqueness of the infinite-connected component \cite[Theorem 4.33]{Gri06}, and in the last, the FKG inequality \cite[Thm 3.8]{Gri06} and the invariance under translation of $\phi^0_{\bbZ^d}$ \cite[Thm 4.19]{Gri06}.  Since this is true for every $x,y\in E$, we deduce that
 \begin{equation}\label{eq:4}
 \liminf_{n\rightarrow\infty} \liminf_{N\to\infty}\ \mathsf S_{N,n}\ge\phi^0_{\mathbb Z^d}[0\longleftrightarrow\infty]^2,
 \end{equation}
For the limit of $\mathsf T_{N,n}$, we use again the random-cluster model. Fix $n$. For every $y$ in $\mathbb T_{N,n}$ at graph distance larger than $2R\le N$ from $0$, the comparison between boundary conditions \cite[Lem 4.14]{Gri06} gives
\begin{align*}
\phi_{\bbT_{N,n}}[x\longleftrightarrow y]&\le \phi_{B_R}^1[0\longleftrightarrow \partial B_R]\times \phi_{B'_R}^1[x\longleftrightarrow \partial B_R']\le \phi_{B_R}^1[0\longleftrightarrow \partial B_R]^2,
\end{align*}
where we recall that $B_R$ is the ball of radius $R$ around $0$ in $\mathbb T_{\infty,n}$ and $B_R'$ is its translate by $y$. Bounding $\phi_{\bbT_{N,n}}[0\longleftrightarrow y]$ using the inequality above when $0$ and $y$ are at a distance larger than $2R$, and otherwise by 1 for remaining values of $y$, gives 
\begin{align*}
  \frac1{|\mathbb T_{N,n}|} \sum_{y\in \mathbb T_{N,n}} \phi_{\bbT_{N,n}}[0\longleftrightarrow y]\le \frac{|B_R|}{|\mathbb T_{N,n}|}+ \phi^1_{B_R}[0\longleftrightarrow \partial B_R]^2.
\end{align*}
Taking the limit as $N$ tends to infinity (recall that $\phi^1_{B_R}$ tends to $\phi^1_{\bbZ^d}$), and then letting $R$ tend to infinity, gives
\begin{equation*}
  \label{eq:13}
 \limsup_{N\to \infty}  \frac1{|\mathbb T_{N,n}|} \sum_{y\in \mathbb T_{N,n}} \langle \sigma_0\cdot\sigma_y\rangle_{\bbT_{N,n}}\le  \phi^1_{\mathbb T_{\infty,n}}[0\longleftrightarrow \infty]^2.
\end{equation*}
Taking the limit as $n$ tends to infinity gives that 
\begin{equation}
  \label{eq:14}
  \limsup_{n\to\infty}\limsup_{N\to\infty} \mathsf T_{N,n}\le \limsup_{n\to \infty} \phi^1_{\mathbb T_{\infty,n}}[0\longleftrightarrow \infty]^2. 
\end{equation}
Let us now turn to $\mathsf U_{N,n}$ on the right-hand side. As $N$ tends to infinity, the torus Green function $\mathsf G_{N,n}( x, y)$ converges to the ``slab'' Green function $\mathsf G_{\infty,n}(x,y)$ associated with the random walk in $\mathbb T_{\infty,n}$ (as a Riemann sum, the right hand side of \eqref{eq:2} converges to the Fourier representation of $\mathsf G_{\infty,n}(x,y)$, this uses that $r\ge 3$ since one needs the walk in the slab to be transient). Then, the limit of $\mathsf G_{\infty,n}(x,y)$ converges as $n$ tends to infinity to the Green function $\mathsf G$ associated to the simple random walk in $\mathbb Z^d$. Therefore
\begin{equation}
  \label{eq:10}
  \lim_{n\to\infty}\lim_{N\to\infty} \mathsf U_{N,n}= \frac1{|E|^2}\sum_{x,y\in E}\mathsf G(x,y).
\end{equation}

We can now plug \eqref{eq:4}, \eqref{eq:14} and \eqref{eq:10} in \eqref{eq:argh} to get that for every $E$,
\begin{equation}
  \label{eq:3}
\phi^0_{\mathbb Z^d}[0\longleftrightarrow\infty]^2- \limsup_{n\to\infty}\phi^1_{\mathbb T_{\infty,n}}[0\longleftrightarrow \infty]^2\le \frac1{|E|^2}\sum_{x,y\in E}\mathsf G(x,y).
\end{equation}

Now, the random walk in $\mathbb Z^d$ is transient, so that $\mathsf G(x,y)$ tends to 0 as $x$ and $y$ become far apart. We deduce that, as $|E|$ tends to infinity, the right-hand side tends to 0. Overall, we find

%

\begin{equation}
  \label{eq:15}
 \phi^0_{\mathbb Z^d}[0\longleftrightarrow\infty]\le \limsup_{n\to\infty}\phi^1_{\mathbb T_{\infty,n}}[0\longleftrightarrow \infty].
\end{equation}
Since $p_c(\bbZ^d)$ is also defined as the supremum of the values of $p$ for which $\phi^0_{\bbZ^d,p}[0\longleftrightarrow\infty]=0$ \cite[(5.4)]{Gri06}, this immediately implies that
$$p_c(\bbZ^d)\ge \limsup_{n\to\infty}p_c(\bbT_{\infty,n})$$
or equivalently by \eqref{eq:h},
$$\beta_c(\mathbb Z^d)\ge \limsup_{n\to\infty} \beta_c(\bbT_{\infty,n}).$$ Since the inequality 
$\beta_c(\mathbb Z^d)\le \liminf \beta_c(\bbT_{\infty,n})$
is given by  Proposition~\ref{prop:1}, this  concludes the proof.
\end{proof}

Note that the estimate given by the infrared bound is absolutely crucial here and that we do not know how to go around it for non-integer values of $q$. This perfectly illustrates that the random-cluster models associated with the Ising and Potts models may be simpler to handle than Bernoulli percolation in some cases. Other instances of such a phenomenon include proofs of conformal invariance in two dimensions for the random-cluster model with cluster-weight $q=2$ on the square lattice \cite{Smi10} (the corresponding result is still open for Bernoulli percolation) or the proof of continuity of the phase transition for the Ising model in dimension 3 and therefore the random-cluster model with $q=2$ in dimension 3 \cite{AizDumSid15} (the corresponding statement is one of the main conjectures in the theory of Bernoulli percolation).

We wish to advertise the following question, which would require to find an argument not relying on the infrared bound:
\begin{question}
Prove that $p_c(\bbZ^r\times(\bbZ/2n\bbZ)^{d-r},q)$ converges to $p_c(\bbZ^d,q)$ for any (non-necessarily integer valued) $q\ge1$ and any $d>r\ge2$.
\end{question}

\section{Comparison with the slab percolation threshold}

In this section, we link the previous result to the Grimmett-Marstrand result \cite{GriMar90}.  Define the random-cluster model on the infinite graph $\bbZ^r\times\llbracket 0,n\rrbracket^{d-r}$. Even though this graph is not transitive, a non-ambiguous notion of a critical point can be defined as before (asking for the smallest value of $p$ for which $0$ is connected to infinity with positive probability). 

\begin{question}\label{question:1} Show that $p_c(\bbZ^r\times\llbracket 0,n\rrbracket^{d-r},q)$ converges to $p_c(\bbZ^d,q)$ for $q\ge1$ and $d>r\ge2$.\end{question}
Even for integers $q\ge2$ and $r\ge3$, this question, first raised in \cite{Pis96}, is open. Note that combined with the following simple proposition, it would imply that finite connected components have exponential tails for $p>p_c(\bbZ^d)$ (when $d=2$, this result follows from duality and \cite{BefDum12}).
\begin{proposition}[{\cite[Thm.~5.104]{Gri06}}]
Fix $q\ge1$ and $d>r\ge2$. If $p_c(\bbZ^r\times\llbracket 0,n\rrbracket^{d-r},q)$ converges to $p_c(\bbZ^d,q)$, then for any $p>p_c(\bbZ^d,q)$, there exists $c=c(p,q)>0$ such that 
$$\phi_{\bbZ^d,p,q}^1[0\longleftrightarrow x,0\not\longleftrightarrow\infty]\le \exp(-c\|x\|).$$
\end{proposition}

Let us conclude that the convergence of $p_c(\bbZ^r\times\llbracket 0,n\rrbracket^{d-r},q)$ to $p_c(\bbZ^d,q)$ is equivalent to the convergence of $\beta_c(\bbZ^r\times\llbracket 0,n\rrbracket^{d-r},q)$ to $\beta_c(\bbZ^d,q)$ as $n$ tends to infinity. This result was proved for $q=2$ (i.e.~for the Ising model) in \cite{Bod05}. In this paper, Bodineau obtained a slightly stronger result related to the {\em slab percolation threshold} $\hat\beta_c(\bbZ^d)$ of the Ising model defined by
$$\hat\beta_c(\bbZ^d):=\inf\big\{\beta\text{ for which }\exists n\ge0\text{ such that } \inf_{N}\inf_{x,y\in\Lambda_{N,n}}\langle\sigma_x\cdot\sigma_y\rangle_{\Lambda(2N,n),\beta,2}>0\big\},$$
where $\Lambda_{N,n}:=\llbracket -N,N\rrbracket^{d-1}\times\llbracket -n,n\rrbracket$ and $\langle\cdot\rangle_{\Lambda(2N,n),\beta,2}$ is the $q=2$ Potts measure on $\Lambda_{N,n}$ with free boundary conditions (we omit the definition here). 

This notion, introduced by Pisztora in \cite{Pis96}, enables one to perform a powerful renormalization scheme to derive a number of properties of the regime $\beta> \hat\beta_c(\bbZ^d)$. Motivated by \cite{Pis96}, we propose to enhance Question~\ref{question:1} into the following one. One can easily extend the notion of the slab percolation threshold to Potts model by replacing $q=2$ with an arbitrary integer $q\ge2$. \medbreak
\begin{question}Show that $\hat\beta_c(\bbZ^d)$ is equal to $\beta_c(\bbZ^d)$ for integer values of $q\ge3$.\end{question}
Let us recall for completeness that despite the fact that the regime $\beta>\hat\beta_c(\bbZ^d)$ is well understood, several fundamental questions remain open even under this additional assumption on $\beta$ (and sometimes even in the case of the Ising model): to mention but a few, the continuity of the magnetization parameter, exponential decay of truncated correlations, description of the translational invariant Gibbs states (see \cite{Bod06} and \cite{Rao17} for the case of the Ising model), etc. Investigating these questions further  is of prime importance.

 \paragraph{Acknowledgments} 
 This research was supported by the IDEX grant from Paris-Saclay and the NCCR SwissMAP. We thank Angelo Ab\"acherli for useful discussions, and Sébastien Martineau and Aran Raoufi for their careful reading of our manuscript and their insightful comments and suggestions.

\bibliographystyle{alpha}

\begin{thebibliography}{{O'D}14}
\bibitem[Aba17]{Aba}
A. Ab\"acherli.
\newblock Local picture and level-set percolation of the Gaussian free field on a large discrete torus.
\newblock {\em arXiv:1707.05935}, 2017.

\bibitem[ADS15]{AizDumSid15}
M.~Aizenman, H.~{Duminil-Copin}, and V.~Sidoravicius.
\newblock Random {C}urrents and {C}ontinuity of {I}sing {M}odel's {S}pontaneous
  {M}agnetization.
\newblock {\em Communications in Mathematical Physics}, 334:719--742, 2015.


\bibitem[BD12]{BefDum12}
V.~Beffara and H.~{Duminil-Copin}.
\newblock The self-dual point of the two-dimensional random-cluster model is
  critical for {$q\geq 1$}.
\newblock {\em Probab. Theory Related Fields}, 153(3-4):511--542, 2012.

\bibitem[Ben13]{ib} I. \ { Benjamini}, {\it Euclidean vs.\@ Graph Metric
  }, Erdos Centennial, Springer Berlin Heidelberg, p.\ 35-57, 2013.

\bibitem[BNP11]{bnp} I.\ { Benjamini}, A.\ { Nachmias} and Y.\ { Peres}, {\it Is critical percolation lo\-cal?}, Probability Theory and Related Fields, vol.\ 149, p.\ 261-269, 2011.

\bibitem[BS96]{bs1} I.\ { Benjamini} and O.\ { Schramm}, {\it Percolation beyond $\bbZ^d$, many questions and a few answers}, Electronic Communication in Probability, vol.\ 1, p.\ 71-82, 1996.

\bibitem[Bis09]{Bis09}
M.~Biskup. \emph{Reflection positivity and phase transitions in lattice spin
  models}, Methods of contemporary mathematical statistical physics, Springer,
  2009, pp.~1--86.

\bibitem[Bod05]{Bod05}
T.~Bodineau. 
\newblock Slab percolation for the {I}sing model, 
\newblock {\em Prob. Theo. Rel. fields}, 132(1):83--118, 2005.

\bibitem[Bod06]{Bod06}
T.~Bodineau. Translation invariant {G}ibbs states for the {I}sing model,
  {\em Prob. Theo. Rel. Fields}, 135(2):153--168, 2006.
  

\bibitem[Dum17]{Dum17}
H.~Duminil-Copin.
\newblock Lectures on the {I}sing and {P}otts models on the hypercubic lattice.
\newblock arXiv:1707.00520, 2017.

\bibitem[DRT17]{DumRaoTas17}
H.~Duminil-Copin, A.~Raoufi, and V.~Tassion.
\newblock Sharp phase transition for the random-cluster and potts models via
  decision trees.
\newblock arXiv:1705.03104, 2017.

\bibitem[DT16]{DumTas16}
H.~{Duminil-Copin} and V.~Tassion.
\newblock A new proof of the sharpness of the phase transition for {B}ernoulli
  percolation and the {I}sing model.
\newblock {\em Communications in {M}athematical {P}hysics}, 343(2):725--745,
  2016.
  
\bibitem[FK72]{ForKas72}
C.~M. Fortuin and P.~W. Kasteleyn.
\newblock On the random-cluster model. {I}. {I}ntroduction and relation to
  other models.
\newblock {\em Physica}, 57:536--564, 1972.

\bibitem[FSS76]{FroSimSpe76}
J.~Fr{{\"o}}hlich, B.~Simon, and Thomas Spencer.
\newblock Infrared bounds, phase transitions and continuous symmetry breaking.
\newblock {\em Comm. Math. Phys.}, 50(1):79--95, 1976.

\bibitem[FS81]{FroSpe81}
J{{\"u}}rg Fr{{\"o}}hlich and Thomas Spencer.
\newblock The {K}osterlitz-{T}houless transition in two-dimensional abelian
  spin systems and the {C}oulomb gas.
\newblock {\em Comm. Math. Phys.}, 81(4):527--602, 1981.


\bibitem[Gri06]{Gri06}
Geoffrey Grimmett.
\newblock {\em The random-cluster model}, volume 333 of {\em Grundlehren der
  Mathematischen Wissenschaften [Fundamental Principles of Mathematical
  Sciences]}.
\newblock Springer-Verlag, Berlin, 2006.

\bibitem[GL14]{gl1} G.\ { Grimmett} and Z.\ { Li}, {\it Locality of connective constants, I. Transitive graphs}, arXiv:1412.0150, 2014.

\bibitem[GL15]{gl2} G.\ { Grimmett} and Z.\ { Li}, {\it Locality of
    connective constants, II. Cayley graphs}, arXiv:1501.00476, 2015.


\bibitem[GM90]{GriMar90}
G.~R.~Grimmett and J.~M.~Marstrand.
\newblock The supercritical phase of percolation is well behaved.
\newblock {\em Proc. Roy. Soc. London Ser. A}, 430(1879):439--457, 1990.

\bibitem[Law13]{Law13}
G.~F.~Lawler. 
\newblock Intersections of random walks
\newblock {\em Modern Birkh\"auser Classics}, Birkh\"auser/Springer, New York,  2013.

\bibitem[DeLS11]{lss} B.\ NB\ { De Lima}, R. { Sanchis}, R. WC {
    Silva}, {\it Critical point and percolation probability in a long
    range site percolation model on $\mathbb Z^d$}, Stochastic
  processes and their applications, vol. 121 (No 9), p. 2043--2048,
  2011. 

\bibitem[MT17]{MarTas}
S.~Martineau and V.~Tassion, \newblock Locality of percolation for abelian Cayley graphs.
\newblock  {\em Annals of Probability} \newblock 45(2):1247--1277, 2017.
\bibitem[Pis96]{Pis96}
A.~Pisztora.
\newblock Surface order large deviations of Ising, Potts and percolation models,
\newblock {\em Prob. Th.
Rel. Fields}
\newblock 104:427--466, 1996.

\bibitem[Pet15]{Pet15}
G{\'a}bor Pete.
\newblock Probability and geometry on groups. lecture notes for a graduate
  course.
\newblock {\em Lecture Notes}, 2015.

\bibitem[Rao17]{Rao17}
A.~Raoufi.
\newblock Translation Invariant Ising Gibbs States: General Setting,
\newblock {\em to appear}, 2017.

\bibitem[Smi10]{Smi10}
Stanislav Smirnov.
\newblock Conformal invariance in random cluster models. {I}. {H}olomorphic
  fermions in the {I}sing model.
\newblock {\em Ann. of Math. (2)}, 172(2):1435--1467, 2010.

\bibitem[SXZ14]{SXZ14}
H.\ { Song}, K.\ N.\ { Xiang} and  S.\ C.\ H.\ { Zhu}. {\it Locality of percolation critical probabilities: uniformly
nonamenable case}, Preprint arXiv:1410.2453, 2014.

\bibitem[Tas14]{Tas14}
Vincent Tassion.
\newblock {\em {Planarity and locality in percolation theory}}.
\newblock Theses, {Ecole normale sup{\'e}rieure de lyon - ENS LYON}, June 2014.

\end{thebibliography}

\end{document}